\documentclass[10pt,a4paper]{amsart}
\usepackage{amsmath}
\usepackage{amsfonts}
\usepackage{amssymb}
\usepackage[american]{babel}
\usepackage[colorlinks=true,allcolors=blue]{hyperref}
\usepackage{stmaryrd}
\usepackage{pgf,tikz}

\title[Addition Theorems in $\mathbb{F}_p$ via the polynomial method]{Addition Theorems in $\mathbb{F}_p$\\ via the polynomial method}
\author{\'Eric Balandraud}
\date{}

\newtheorem{defi}{Definition}
\newtheorem{theo}{Theorem}
\newtheorem*{conj}{Conjecture}
\newtheorem{prop}{Proposition}
\newtheorem{rem}{Remark}

\begin{document}

\maketitle

\begin{abstract}
In this article, we use the Combinatorial Nullstellensatz to give new proofs of the Cauchy-Davenport, the Dias da Silva-Hamidoune and to generalize a previous addition theorem of the author. Precisely, this last result proves that for a set $A\subset \mathbb{F}_p$ such that $A\cap(-A)=\emptyset$ the cardinality of the set of subsums of at least $\alpha$ pairwise distinct elements of $A$ is:
\[|\Sigma_\alpha(A)|\geq\min\left\{p,\frac{|A|(|A|+1)}{2}-\frac{\alpha(\alpha+1)}{2}+1\right\},\]
the only cases previously known were $\alpha\in\{0,1\}$.

The Combinatorial Nullstellensatz is used, for the first time, in a direct and in a reverse way. The direct (and usual) way states that if some coefficient of a polynomial is non zero then there is a solution or a contradiction. The reverse way relies on the coefficient formula (equivalent to the Combinatorial Nullstellensatz). This formula gives an expression for the coefficient as a sum over any cartesian product.

For these three addition theorems, some arithmetical progressions (that reach the bounds) will allow to consider cartesian products such that the coefficient formula is a sum all of whose terms are zero but exactly one. Thus we can conclude the proofs without computing the appropriate coefficients.
\end{abstract}

\section{Introduction}

In this article, $p$ is always a prime number, given two non-empty subsets $A$ and $B$ of $\mathbb{F}_p$, we denote their sumset $A+B=\{a+b\mid a\in A,\ b\in B\}$.

The first addition theorem in $\mathbb{F}_p$ is the Cauchy-Davenport theorem.
\begin{theo}[Cauchy-Davenport \cite{Ca,Da1,Da2}]
Let $A$ and $B$ be two non empty subsets of $\mathbb{F}_p$, then:
\[|A+B|\geqslant\min\left\{p,|A|+|B|-1\right\}.\]
\end{theo}

In the seminal article \cite{Al}, Alon described the Combinatorial Nullstellensatz and the polynomial method that relies on it (described in section $\ref{PM}$). The method allows to prove that a combinatorial problem has a solution or a contradiction, just by computing a certain coefficient in a polynomial. The combinatorial problem is reduced to a computation problem. The Cauchy-Davenport is one of the first example developed in this article. The binomial theorem is the key point that allows to prove that the proper coefficient is non zero.

Surprisingly a slight variation of the definition of the sumset has revealed itself much more difficult to tackle. For two subsets  $A$ and $B$ of $\mathbb{F}_p$, we define their restricted sumset: $A\dot{+}B=\{a+b\mid a\in A,\ b\in B,\ a\neq b\}$. In $1964$, Erd\H os and Heilbronn made the following famous conjecture:
\begin{conj}[Erd\H os-Heilbronn] Let $A\subset\mathbb{F}_p$, then:
\[|A\dot{+}A|\geqslant\min\left\{p,2|A|-3\right\}.\]
\end{conj}

The first proof follows from the following generalization in $1994$ by Dias da Silva and Hamidoune, introducing the $h$-fold restricted sumset:
\begin{defi} Let $A\subset \mathbb{F}_p$ and $h\in[0,|A|]$, we denote $h^\wedge A$ the set of subsums of $h$ pairwise distinct elements of $A$:
\[h^\wedge A=\{a_1+\dots+a_h\mid a_i\in A,\ a_i\neq a_j\}.\]
\end{defi}

\begin{theo}[Dias da Silva, Hamidoune \cite{DH}] Let $A\subset \mathbb{F}_p$. For a natural integer $h\in[0,|A|]$,
\[|h^\wedge A|\geqslant\min\{p,h(|A|-h)+1\}.\]
\end{theo}
Their proof relies on exterior algebras. A second proof of this result was given the following year by Alon, Nathanson and Rusza. They applied the Combinatorial Nullstellensatz \cite{ANR1,ANR2}. To prove that the proper coefficient is non zero, they consider another combinatorial interpretation of it through strict ballot number.

Following this method, the author could prove a further statement considering the set of all subsums.
\begin{defi} Let $A\subset \mathbb{F}_p$, we denote its set of subsums by:
\begin{align*}
\Sigma(A)\  & =\left\{\sum_{x\in I}x\mid\emptyset\subset I\subset A\right\}=\bigcup_{h=0}^{|A|}(h^\wedge A)\\
\intertext{and we also denote its set of non-trivial subsums by:}
\Sigma^*(A) & =\left\{\sum_{x\in I}x\mid\emptyset\subsetneq I\subset A\right\}=\bigcup_{h=1}^{|A|}(h^\wedge A).
\end{align*}
\end{defi}

For the following result the computation of the coefficient relied on determinants of binomial coefficients: binomial determinants considered in the work of Gessel and Viennot.
\begin{theo}[Balandraud \cite{EB}]\label{EB} Let $A\subset\mathbb{F}_p$, such that $A\cap(-A)=\emptyset$. We have
\begin{align*}
|\Sigma(A)|\ & \geqslant\min\left\{p,\frac{|A|(|A|+1)}{2}+1\right\},\\
|\Sigma^*(A)| & \geqslant\min\left\{p,\frac{|A|(|A|+1)}{2}\right\}.
\end{align*}
\end{theo}

Among other the applications of this result are algebraic invariants: Noether number or Davenport constant variations \cite{CZ,OPSS,WS}. Many of these applications would consider the bound on $\Sigma^*(A)$ in order to ensure the existence of a \emph{non trivial} zero-subsum of $A$. For these problems it is also of interest to consider subsums with a larger restriction on the number of terms. This is the aim of the last and new result of this article. We define:

\begin{defi} Let $A\subset \mathbb{F}_p$, we denote $\Sigma_\alpha(A)$ the set of subsums of at least $\alpha$ pairwise distinct elements of $A$ and $\Sigma^\alpha(A)$  the set of subsums of at most $|A|-\alpha$ pairwise distinct elements of $A$.
\begin{align*}
\Sigma_\alpha(A) & = \{a_1+\dots+a_k\mid a_i\in A,\ \alpha\leq k\leq |A|,\ a_i\neq a_j\}  & & =\bigcup_{k=\alpha}^{|A|}(k^\wedge A)\\
\Sigma^\alpha(A) & = \{a_1+\dots+a_k\mid a_i\in A,\ 0\leq k\leq |A|-\alpha,\ a_i\neq a_j\} & & =\bigcup_{k=0}^{|A|-\alpha}(k^\wedge A).
\end{align*}
\end{defi}

These sets of subsums satisfy the following elementary properties:

\begin{itemize}
\item Whenever $\alpha\in\{0,1\}$, one has $\Sigma_0(A)=\Sigma^0(A)=\Sigma(A)$ and $\Sigma_1(A)=\Sigma^*(A)$.
\item Whatever $\alpha$, one has the symmetry: $\Sigma_\alpha(A)=\left(\sum_{a\in A}a\right)-\Sigma^\alpha(A)$, what implies that $|\Sigma_\alpha(A)|=|\Sigma^\alpha(A)|$.
\item Whenever $\alpha\leq\alpha'$ one has $\Sigma_{\alpha'}(A)\subset\Sigma_\alpha(A)$.
\end{itemize}

The generalization of Theorem \ref{EB} is:

\begin{theo}\label{Main} Let $A\subset\mathbb{F}_p$, such that $A\cap(-A)=\emptyset$. For any natural integer $\alpha\in[0,|A|]$, we have:
\[|\Sigma_\alpha(A)|=|\Sigma^\alpha(A)|\geqslant\min\left\{p,\frac{|A|(|A|+1)}{2}-\frac{\alpha(\alpha+1)}{2}+1\right\}.\]
\end{theo}

Before the proof, we can make the following remarks:

\begin{itemize}
\item Whenever $\alpha\in\{0,1\}$, this is exactly Theorem \ref{EB}.
\item This bound is sharp since for $A=[1,d]$, one has:
\[\Sigma^\alpha(A)=\left[0,\frac{d(d+1)}{2}-\frac{\alpha(\alpha+1)}{2}\right].\]
of cardinality exactly $\min\left\{p,\frac{|A|(|A|+1)}{2}-\frac{\alpha(\alpha+1)}{2}+1\right\}$.
\end{itemize}

The article is organized as follows: In a first section, we explain the method. We state the Combinatorial Nullstellensatz, the coefficient formula and the new proofs of the Cauchy-Davenport and Dias da Silva-Hamidoune theorems. The novelty in these proofs, is that there would be no need to compute the coefficients. In a second section, the proof of Theorem \ref{Main} is given. It follows the steps of the method of the first section. In the last section, we discuss the problem of the sets of subsums with upper and lower bound on the number of terms. It appears surprisingly that the problem with a double bound is of a different nature than the three previous ones.

\section{Rewriting the polynomial proofs of Cauchy-Davenport and Dias da Silva-Hamidoune theorems}\label{PM}

\subsection{The polynomial method}

The Combinatorial Nullstellensatz is a result that generalizes to multivariate polynomials the fact that an univariate polynomial of degree $d$ cannot vanish on $d+1$ points.

\begin{theo}[Combinatorial Nullstellensatz \cite{Al}]\label{CNSS} Let $\mathbb{F}$ be any field and $P(\underline{X})\in\mathbb{F}[X_1,\dots,X_d]$. If $P$ has total degree $k_1+\dots+k_d$
and its coefficient of the monomial $\prod_{i=1}^dX_i^{k_i}$ is non-zero, then whatever is the family $(A_1,\dots,A_d)$ of subsets of $\mathbb{F}$ satisfying $|A_i|>k_i$, there is a point $\underline{a}\in A_1\times\dots\times A_d$ such that
\[P(\underline{a})\neq 0.\]
\end{theo}
This theorem has lead to numerous proofs of combinatorial conjectures and new proofs in many mathematical fields. It is called Combinatorial Nullstellensatz because another formulation of it gives a generating family of the ideal of polynomial that vanishes on a cartesian product. The previously stated formulation is a criterion for a polynomial not to belong to this ideal.

Applying the polynomial method (the one that relies on the Combinatorial Nullstellensatz) on a combinatorial problem consists in defining a (big enough) cartesian product and a polynomial of small degree, so that the Combinatorial Nullstellensatz, will assert that there is a solution or a contradiction provided that a specific coefficient is nonzero. The combinatorial problem is then reduced to the computation problem of the appropriate coefficient.

In the three problems treated in this article, we will not need to compute the coefficient. We use the coefficient formula proved independently by Karasev-Petrov and by L\'ason, it is equivalent to the Combinatorial Nullstellensatz:
\begin{theo}(Coefficient formula \cite{KP,Las}) Let $P\in\mathbb{F}[X_1,\dots,X_d]$ be a polynomial of degree $k_1+\dots+k_d$ and any family of sets $A_i$, with $|A_i|=k_i+1$, denoting $g_i(X)=\prod_{a\in A_i}(X-a)$, then the coefficient of the monomial $\prod_{i=1}^dX_i^{k_i}$ in the expansion of $P$ is
\[\sum_{\underline{b}\in\prod_{i=1}^d A_i}\frac{P(\underline{b})}{\prod_{i=1}^dg_i'(b_i)}.\]
\end{theo}

In \cite{KP}, Karasev and Petrov gave a new proof of Dyson's conjecture thanks to this formula using an auxiliary polynomial and cartesian product.

We will use the coefficient formula for some well chosen sets to prove that the wanted coefficient is not zero. This does not require to compute the coefficient. The coefficient formula will provide
an expression of the specified coefficient as a sum, all of whose terms are zero but exactly one.

In our context the bound is tight and reached by some arithmetical progressions. The way to choose the auxiliary polynomial and cartesian product will be to consider the same constructions for these arithmetical progressions. In conclusion, our method is a way to understand why these bounds are reached by these arithmetical progressions via a kind of algebraic comparison.

\subsection{A proof of the Cauchy-Davenport theorem}

\begin{proof}
Let us consider two non empty subsets $A$ and $B$ of $\mathbb{F}_p$, of respective cardinality, $|A|=n$ and $|B|=m$. Define $\delta=\max\{0,n+m-1-p\}$. Since $\max\{n,m\}\leq p$, one has $\delta<\min\{n,m\}$.

We will prove the theorem by contradiction. Let us suppose that $|A+B|<\min\{p,n+m-1\}$, then consider a set $C$ of cardinality $|C|=\min\{p-1,n+m-2\}=n+(m-\delta)-2<p$ that contains $A+B$.

Define the polynomial
\[P(X,Y)=\prod_{x\in C}(X+Y-x).\]

By definition, $P$ vanishes on the cartesian product $A\times B$. We have $\deg(P)=|C|=(n-1)+(m-\delta-1)$.

Using the Combinatorial Nullstellensatz, to obtain a contradiction, it suffices to prove that the coefficient $c_{n-1,(m-\delta)-1}$ of $X^{n-1}Y^{(m-\delta)-1}$ is not zero.

Now consider the sets $A'=[1,n]$ and $B'=[1,(m-\delta)]$, one has $A'+B'=[2,n+(m-\delta)]$. We also consider the polynomial
\[Q(X,Y)=\prod_{x=2}^{n+(m-\delta)-1}(X+Y-x).\]
The polynomial $Q(X,Y)$ is defined similarly as $P(X,Y)$ on a set $C'=[2,n+(m-\delta)-1]$ of cardinality $|C'|=n+(m-\delta)-2=|C|$. Since $|C'|<p$, the elements of $[2,n+(m-\delta)]$ are pairwise distinct modulo $p$.
The two polynomial $P$ and $Q$ have the same coefficients of maximal degree, in particular they have the same coefficient $c_{n-1,(m-\delta)-1}$ of the monomial $X^{n-1}Y^{(m-\delta)-1}$.

We can use the coefficient formula on the sets $A'$ and $B'$ to find this coefficient in $Q$. The key point of this proof is the fact that the polynomial $Q$ vanishes on all the element of $A'\times B'$ but one: $Q(n,(m-\delta))\neq 0$. Therefore the coefficient is
\begin{align*}
c_{n-1,(m-\delta)-1}= & \sum_{(a,b)\in A'\times B'}\frac{Q(a,b)}{\prod_{a'\in A'\setminus\{a\}}(a-a')\prod_{b'\in B'\setminus\{b\}}(b-b')}\\
= & \frac{Q(n,(m-\delta))}{\prod_{i=1}^{n-1}(n-i)\prod_{i=1}^{(m-\delta)-1}((m-\delta)-i)}\neq 0.
\end{align*}

The expression as a sum that contains exactly one non-zero term suffices to assert that it is non zero.
\end{proof}

In this case, the computation is easy and the previous formula also proves that $c_{n-1,(m-\delta)-1}=\binom{n+(m-\delta)-2}{n-1}$.

\subsection{A proof of the Dias da Siva-Hamidoune theorem}

\begin{proof}
Consider a subset $A=\{a_1,\dots,a_d\}$ of $\mathbb{F}_p$ and $h\in[0,d]$.

We will prove the theorem by contradiction. Suppose that $h^\wedge A\subset C$, with $|C|=\min\{p-1,h(d-h)\}$.

Let us denote $\delta=\max\{0,h(d-h)+1-p\}$, this implies that $|C|=h(d-h)-\delta<p$. Since $h^\wedge(A\setminus\{a_d\})\subset h^\wedge A$, one can consider that $h((d-1)-h)+1<p$, what implies that $\delta<h$.

Let us consider the polynomial of $P_{d,h,\delta}(\underline{X})\in\mathbb{F}_p[X_1,\dots,X_h]$:
\[P_{d,h,\delta}(\underline{X})=\prod_{x\in C}(X_1+X_2+\dots+X_h-x)\prod_{1\leq i<j\leq h}(X_j-X_i).\]
By definition of $C$, $P_{d,h,\delta}$ vanishes on the whole cartesian product $A^h$. In our context, we will consider the sub-cartesian product $A_1\times\dots\times A_h$, where:
\[\begin{array}{lcl}
A_1            & = & \{a_1,\dots,a_{d-h}\} \\
\ \vdots       & &   \ \ \ \vdots\hspace{1.8cm} \ddots\\
A_\delta       & = & \{a_1,\hspace{.6cm}\dots\hspace{.6cm},a_{d-h+\delta-1}\} \\
A_{\delta+1}   & = & \{a_1,\hspace{1cm}\dots\hspace{1cm},a_{d-h+\delta+1}\} \\
\ \vdots       & &   \ \ \ \vdots\hspace{3.7cm} \ddots \\
A_h            & = & \{a_1,\hspace{1.8cm}\dots\hspace{1.8cm},a_d\} \\
\end{array}\]

On the first hand, one has
\begin{align*}
\deg(P)= & |C|+\frac{h(h-1)}{2}\\
 = & h(d-h)+\frac{h(h-1)}{2}-\delta\\
 = & dh-\frac{h(h+1)}{2}-\delta,
\end{align*}
and on the other hand $\sum_{i=1}^h(|A_i|-1)=dh-\frac{h(h+1)}{2}-\delta$.

Thanks to the Combinatorial Nullstellensatz, to obtain a contradiction, it suffices to prove that the coefficient $c_{d,h,\delta}$ of the monomial $\prod_{i=1}^hX_i^{|A_i|-1}=\prod_{i=1}^\delta X_i^{d-h+i-2}\prod_{i=\delta+1}^h X_i^{d-h+i-1}$ is not zero.

We now consider the same construction for the set $B=[1,d]$ that satisfy $h^\wedge B=\left[\frac{h(h+1)}{2},\frac{d(d+1)}{2}-\frac{(d-h)(d-h+1)}{2}\right]$ of cardinality $|h^\wedge B|=\min\{p,h(d-h)+1\}$.

Let us consider the cartesian product $B_1\times\dots\times B_h$:
\[\begin{array}{lcl}
B_1            & = & \{1,\dots,(d-h)\} \\
\ \vdots       & &  \ \ \vdots\hspace{1.8cm} \ddots\\
B_\delta       & = & \{1,\hspace{.6cm}\dots\hspace{.6cm},(d-h+\delta-1)\} \\
B_{\delta+1}   & = & \{1,\hspace{1cm}\dots\hspace{1cm},(d-h+\delta+1)\} \\
\ \vdots       & &  \ \ \vdots\hspace{3.7cm} \ddots \\
B_h            & = & \{1,\hspace{2cm}\dots\hspace{2cm},d\}.
\end{array}\]

We also define the set $R=\left[\frac{h(h+1)}{2},\frac{d(d+1)}{2}-\frac{(d-h)(d-h+1)}{2}-\delta-1\right]$. (Since $h(d-h)-\delta<p$, the elements of $R$ are pairwise distinct modulo $p$ and do not cover $\mathbb{F}_p$.) Finally, we define the polynomial
\[Q_{d,h,\delta}(\underline{X})=\prod_{x\in R}(X_1+X_2+\dots+X_h-x)\prod_{1\leq i<j\leq h}(X_j-X_i).\]
Since $|R|=h(d-h)-\delta=|C|$, the two polynomials $Q_{d,h,\delta}$ and $P_{d,h,\delta}$ have same degree.  Moreover they differ only by constants in their linear factors, so they have the same coefficients of maximal degree. In particular, they share the have the same coefficient
$c_{d,h,\delta}$ of the monomial $\prod_{i=1}^\delta X_i^{d-h+i-2}\prod_{i=\delta+1}^h X_i^{d-h+i-1}$.

If we consider the sums $b_1+\dots+b_h$ of pairwise different values $b_i\in B_i$, one can reach any value in $\left[\frac{h(h+1)}{2},\frac{d(d+1)}{2}-\frac{(d-h)(d-h+1)}{2}-\delta\right]$. Only one of the values is missing in $R$, namely $\frac{d(d+1)}{2}-\frac{(d-h)(d-h+1)}{2}-\delta$ and this value is uniquely reached by the sum $(d-h)+\dots+(d-h+\delta-1)+(d-h+\delta+1)+\dots+d$. This implies that there is only one point $\underline{b}^*$ in the cartesian product $B_1\times\dots\times B_h$ such that $Q_{d,h,\delta}(\underline{b}^*)\neq 0$. Using the coefficient formula, one get that:

\[c_{d,h,\delta}=\sum_{\underline{b}\in \prod B_i}\frac{Q_{d,h,\delta}(\underline{b})}{\prod g'_i(b_i)}= \frac{Q_{d,h,\delta}(\underline{b}^*)}{\prod_{i=1}^dg_i'(b^*_i)} \neq 0,\]
where \[\underline{b}^*=(\underbrace{(d-h),\dots,(d-h+\delta-1)}_{i=1..\delta},\underbrace{(d-h-\delta+1),\dots,d}_{i=\delta+1..h}).\]

This coefficient is therefore different from zero and the proof is complete.
\end{proof}

\begin{rem}
The computation of the coefficient $c_{d,h,\delta}$ can be proceed to a closed expression, it is done in proposition \ref{comput1} in the annex of this article.
\end{rem}

\section{Sets of subset sums whose number of terms is bounded}

We proceed now to the proof of theorem \ref{Main}:

\begin{proof}

Whenever $p=2$, the hypothesis $A\cap(-A)=\emptyset$ is impossible for a non-empty subset, so from now on $p$ is an odd prime. Consider that the set is $A=\{2a_1,2a_2,\dots,2a_d\}$, so $|A|=d$ and denote $m=\sum_{i=1}^da_i$.

We prove the theorem by contradiction. Suppose that $|\Sigma^\alpha(A)|<\min\left\{p,\frac{d(d+1)}{2}-\frac{\alpha(\alpha+1)}{2}+1\right\}$, and consider a set $C$, such that $\Sigma^\alpha(A)\subset C$, with $|C|=\min\left\{p-1,\frac{d(d+1)}{2}-\frac{\alpha(\alpha+1)}{2}\right\}$.

Denote:
\[\delta=\max\left\{0,\frac{d(d+1)}{2}-\frac{\alpha(\alpha+1)}{2}-(p-1)\right\}.\]
So that $|C|=\frac{d(d+1)}{2}-\frac{\alpha(\alpha+1)}{2}-\delta<p$.

Since one has $\Sigma^{\alpha+1}(A)\subset\Sigma^\alpha(A)$. One can consider that $\frac{d(d+1)}{2}-\frac{(\alpha+1)(\alpha+2)}{2}+1<p$. This implies that $\delta\leq\alpha$.

\bigskip

We define the polynomial:
\begin{align*}
P_{d,\alpha,\delta}(\underline{X})= & \prod_{x\in C}(X_1+\dots+X_d+m-x)\prod_{1\leq i<j\leq d}(X_j-X_i)\prod_{\substack{1\leq i<j\leq d\\\textrm{and}\ j>\alpha}}(X_j+X_i)\\
\end{align*}
This polynomial has degree
\begin{align*}
\deg(P_{d,\alpha,\delta})= & \left(\frac{d(d+1)}{2}-\frac{\alpha(\alpha+1)}{2}-\delta\right)+\left(\frac{d(d-1)}{2}\right)
+\left(\underbrace{(d-\alpha)\alpha}_{j>\alpha,\ \textrm{and}\ i\leq\alpha}+\underbrace{\frac{(d-\alpha)(d-\alpha-1)}{2}}_{\alpha<i<j}\right)\\
= & d^2+\frac{d(d-1)}{2}-\alpha^2-\delta.
\end{align*}

Let us consider the sets:
\[\begin{array}{lcll}
A_1            & = & \{-a_d,\dots,-a_{\alpha+1}\} & \\
\ \vdots       & &   \ \ \ \vdots\hspace{1.8cm} \ddots & \\
A_\delta       & = & \{-a_d,\hspace{.3cm}\dots\hspace{.3cm},-a_{\alpha-\delta+2}\} & \\
A_{\delta+1}   & = & \{-a_d,\hspace{1cm}\dots\hspace{1cm},-a_{\alpha-\delta}\} & \\
\ \vdots       & &   \ \ \ \vdots\hspace{3.7cm} \ddots & \\
A_\alpha        & = & \{-a_d,\hspace{1.5cm}\dots\hspace{1.5cm},-a_1\} & \\
A_{\alpha+1}    & = & \{ -a_d,\hspace{1.5cm}\dots\hspace{1.5cm},-a_1, & a_1,\hspace{.5cm}\dots\hspace{.5cm}, a_{\alpha+1}\}\\
\ \vdots       & &  \ \ \ \vdots\hspace{4.2cm} \vdots &\ \vdots \hspace{3.2cm}\ddots\\
A_d   & = & \{ -a_d,\hspace{1.5cm}\dots\hspace{1.5cm},-a_1, & a_1,\hspace{1.5cm}\dots\hspace{1.5cm}, a_d\}\\
\end{array}\]
Moreover, one also have:
\begin{align*}
\sum_{i=1}^d(|A_i|-1)= & \left(d^2-\frac{\alpha(\alpha-1)}{2}\right)+\left(d(d-\alpha)-\frac{(d-\alpha)(d-\alpha-1)}{2}\right)-d-\delta\\
= & d(2d-\alpha-1)-\left(\alpha^2+\frac{d(d-1)}{2}-d\alpha\right)-\delta\\
= & d^2+\frac{d(d-1)}{2}-\alpha^2-\delta.
\end{align*}

Whatever is the element of the cartesian product, if the two last factors of $P_{d,\alpha,\delta}$ do not vanish then it consists of a sum of the type $\pm a_1\pm a_2\dots\pm a_d$, which has at least $\alpha$ negative signs. So $\pm a_1\pm a_2\dots\pm a_d+m$ is a sum of at most $d-\alpha$ elements of $A$ and the first factor vanishes. In conclusion, $P_{d,\alpha,\delta}$ vanishes on the whole cartesian product $\prod_{i=1}^d A_i$.

To obtain a contradiction thanks to the Combinatorial Nullstellensatz, it suffices to prove that the coefficient $c_{d,\alpha,\delta}$ of the following monomial is non zero:
\[\prod_{i=1}^dX_i^{|A_i|-1}=\left(\prod_{i=1}^\delta X_i^{d-\alpha+i-2}\right)\left(\prod_{i=\delta+1}^\alpha X_i^{d-\alpha+i-1}\right)  \left(\prod_{i=\alpha+1}^dX_i^{d+i-1}\right).\]

Let us now consider the same construction for the set $B=2.[1,d]$: one has
\[\Sigma^\alpha(B)=2.\left[0,\frac{d(d+1)}{2}-\frac{\alpha(\alpha+1)}{2}\right],\] of cardinality $|\Sigma^\alpha(B)|=\frac{d(d+1)}{2}-\frac{\alpha(\alpha+1)}{2}+1$.
Define the sets:
\[\begin{array}{lcll}
B_1         & = & \{-d,\dots,-(\alpha+1)\} & \\
\ \vdots    & &   \ \ \ \vdots\hspace{1.8cm} \ddots & \\
B_\delta     & = & \{-d,\hspace{.3cm}\dots\hspace{.3cm},-(\alpha-\delta+2)\} & \\
B_{\delta+1} & = & \{-d,\hspace{1cm}\dots\hspace{1cm},-(\alpha-\delta)\} & \\
\ \vdots  & &   \ \ \ \vdots\hspace{3.7cm} \ddots & \\
B_\alpha    & = & \{-d,\hspace{1.5cm}\dots\hspace{1.5cm},-1\} & \\
B_{\alpha+1} & = & \{-d,\hspace{1.5cm}\dots\hspace{1.5cm},-1 & 1,\hspace{.5cm}\dots\hspace{.5cm},\alpha+1\}\\
\ \vdots  & & \ \ \ \vdots\hspace{4.1cm} \vdots  &  \vdots \hspace{2.5cm} \ddots \\
B_d & = & \{-d,\hspace{1.5cm}\dots\hspace{1.5cm},-1 & 1,\hspace{1cm}\dots\hspace{1.3cm}, d\}\\
\end{array}\]

Let us denote $m'=\frac{d(d+1)}{2}$ and $R=\left[0,\frac{d(d+1)}{2}-\frac{\alpha(\alpha+1)}{2}-\delta-1\right]$. Since $\frac{d(d+1)}{2}-\frac{\alpha(\alpha+1)}{2}-\delta<p$, the elements of $R$ are pairwise distinct modulo $p$ and do not cover $\mathbb{F}_p$.

define the polynomial:
\[Q_{d,\alpha,\delta}(\underline{X})=\prod_{x\in R}(X_1+\dots+X_d+m'-x)\prod_{1\leq i<j\leq d}(X_j-X_i)\prod_{\substack{1\leq i<j\leq d\\\textrm{and}\ j>\alpha}}(X_j+X_i)\]

Since $|R|=\frac{d(d+1)}{2}-\frac{\alpha(\alpha+1)}{2}-\delta=|C|$, the two polynomials $Q_{d,\alpha,\delta}$ and $P_{d,\alpha,\delta}$ have same degree. Moreover they differ only by constants in their linear factors, so they have the same coefficients of maximal degree. In particular, they have the same coefficient of the monomial $\left(\prod_{i=1}^\delta X_i^{d-\alpha+i-2}\right)\left(\prod_{i=\delta+1}^\alpha X_i^{d-\alpha+i-1}\right)\left(\prod_{i=\alpha+1}^dX_i^{d+i-1}\right)$.

If we consider all the sums $b_1+\dots+b_d+m'$ where $b_i\in B_i$ and
\[\prod_{1\leq i<j\leq d}(b_j-b_i)\prod_{\substack{1\leq i<j\leq d\\\textrm{and}\ j>\alpha}}(b_j+b_i)\neq 0,\] one can reach any value in $\left[0,\frac{d(d+1)}{2}-\frac{\alpha(\alpha+1)}{2}-\delta\right]$. Only one value for the sum does miss in $R$, $\frac{d(d+1)}{2}-\frac{\alpha(\alpha+1)}{2}-\delta$, and there is only one element, whose coordinates are pairwise neither equal nor opposite in this cartesian product and that reaches this value. It implies that there is only one point $\underline{b}^*$ in the cartesian product $B_1\times\dots\times B_d$ such that $Q_{d,\alpha,\delta}(\underline{b}^*)\neq 0$. Using the coefficient formula, one
\[c_{d,\alpha,\delta}=\sum_{\underline{b}\in \prod B_i}\frac{Q_{d,\alpha,\delta}(\underline{b})}{\prod g'_i(b_i)}=\frac{Q_{d,\alpha,\delta}(\underline{b}^*)}{\prod_{i=1}^dg_i'(b_i^*)} \neq 0,\]
where \[\underline{b}^*=(\underbrace{-(\alpha+1),\dots,-(\alpha-\delta+2)}_{i=1..\delta},\underbrace{-(\alpha-\delta),\dots,-1}_{i=\delta+1..\alpha},\alpha+1-\delta,\underbrace{\alpha+2,\dots,d}_{i=\alpha+2..d})\]

This coefficient is therefore different from zero, what concludes the proof.
\end{proof}

\begin{rem} The value of $c_{d,\alpha,\delta}$ can be compute from this formula. It is written in proposition \ref{comput2} in the annex of this article. 
\end{rem}

\section{The trouble in the consideration of a double bound}

It seems natural at this point to define the sets of subsums whose number of terms are doubly bounded:

\begin{defi} Let $A\subset \mathbb{F}_p$, we denote $\Sigma_\alpha^\beta(A)$ the set of subsums of at least $\alpha$ and at most $|A|-\beta$ pairwise distinct elements of $A$
\[\Sigma_\alpha^\beta(A)= \{a_1+\dots+a_k\mid a_i\in A,\ \alpha\leq k\leq |A|-\beta,\ a_i\neq a_j\}=\bigcup_{k=\alpha}^{|A|-\beta}(k^\wedge A).\]
\end{defi}

At first glance, one could think that for a set $A\subset \mathbb{F}_p$ such that $A\cap(-A)=\emptyset$ the minimal cardinality of such a set of subsums is again reached on an arithmetical progression of type $[1,d]$, and so that the cardinality of $|\Sigma_\alpha^\beta(A)|$ would be at least:
\[\min\left\{p,\frac{|A|(|A|+1)}{2}-\frac{\alpha(\alpha+1)}{2}-\frac{\beta(\beta+1)}{2}+1\right\}.\]

This does not hold and several counterexamples can be given: 

Let $k\geq 3$ and consider the set $A=\{1,-2,3,\dots,k\}$, then one has:
\begin{align*}
\Sigma_1^1(A)= & \left\{-2,-1,1,2,\dots,\frac{k(k+1)}{2}-5,\frac{k(k+1)}{2}-3,\frac{k(k+1)}{2}-2\right\},\\
\Sigma_2^1(A)= & \left\{-1,1,2,\dots,\frac{k(k+1)}{2}-5,\frac{k(k+1)}{2}-3,\frac{k(k+1)}{2}-2\right\}.
\end{align*}
Considered in $\mathbb{Z}$, one has $|\Sigma_1^1(A)|=\frac{k(k+1)}{2}-1=\frac{k(k+1)}{2}-1-1+1$ and $|\Sigma_2^1(A)|=\frac{k(k+1)}{2}-2=\left(\frac{k(k+1)}{2}-3-1+1\right)+1$.

So whenever $\frac{k(k+1)}{2}-4=p$, one has
\begin{align*}
|\Sigma_1^1(A)|= & p-1=\frac{k(k+1)}{2}-3<\min\left\{p,\frac{k(k+1)}{2}-1-1+1\right\},\\
|\Sigma_2^1(A)|= & p-1=\frac{k(k+1)}{2}-3<\min\left\{p,\frac{k(k+1)}{2}-3-1+1\right\}.
\end{align*}

It is conjectured (but not formally known) that there is an infinite number of couples $(k,p)$ such that $p=\frac{k(k+1)}{2}-4\in \mathbb{P}$. Here follows the list of those with $p<1000$:
\[(5,11),(6,17),(9,41),(14,101),(17,149),(18,167),\]\[(21,227),(26,347),(29,431),(30,461),(33,557),(41,857).\]

Conversely, it can be seen that, for some other prime numbers, the conjecture is true. It implies that the problem is of a different nature from the Cauchy-Davenport, Dias da Silva-Hamidoune theorems and theorem \ref{Main}. These three theorems can be called universal, since the bound is universal in $p$, the cardinality of the set (and their parameters).

However, for this problem, it is still possible to define a polynomial and a cartesian product that would lead to a proof of the bound, provided a specified coefficient is non zero. Of course, since counterexamples are known, for some values of the parameters $d,\alpha,\beta,p$, the specified coefficient will be zero. The computations of these coefficients lead to the idea that the previous counterexamples are the only ones possible. What can be summarized in the following conjecture:

\begin{conj} Let $p$ be a prime number and $A\subset \mathbb{F}_p$ such that $A\cap(-A)=\emptyset$, then
\[|\Sigma_\alpha^\beta(A)|\geqslant\min\left\{p,\frac{|A|(|A|+1)}{2}-\frac{\alpha(\alpha+1)}{2}-\frac{\beta(\beta+1)}{2}+1\right\},\]
unless $A=\lambda.\{1,-2,3,\dots,k\}$, with $\lambda\in\mathbb{F}_p^*$, $\frac{k(k+1)}{2}=p+4$ and $(\alpha,\beta)\in\{(1,1),(1,2),(2,1)\}$.
\end{conj}

\section*{Annex: Computation of the coefficients}

In this annex, we denote $n!!=\prod_{i=0}^{n-1}i!$, the product of the $n$ first factorials. It is an unusual notation, but it satisfies the nice property $\prod_{1\leq i<j\leq n}(j-i)=n!!$.

\subsection{The coefficient involved in the proof of the Dias da Silva-Hamidoune theorem}

\begin{prop}\label{comput1} The coefficient involved in the proof of the Dias da Silva-Hamidoune theorem is:
\[c_{d,h,\delta}=(h(d-h))!\frac{\binom{d-h+\delta-1}{\delta}\binom{h}{\delta}}{\binom{h(d-h)}{\delta}}\frac{h!!(d-h)!!}{d!!}.\]
\end{prop}

\begin{proof}

The computation of the coefficient can be continued:
\[c_{d,h,\delta}=\frac{Q_{d,h,\delta}(\underline{b}^*)}{\prod_{i=1}^dg_i'(b^*_i)},\]
where \[\underline{b}^*=(\underbrace{(d-h),\dots,(d-h+\delta-1)}_{i=1..\delta},\underbrace{(d-h-\delta+1),\dots,d}_{i=\delta+1..h}).\]

Since $g'_i(b_i)=(|B_i|-1)!=(d-h+i-2)!$ if $i\leq \delta$ and $g'_i(b_i)=(|B_i|-1)!=(d-h+i-1)!$ if $i>\delta$, the multinomial sum gives $|R|!=(h(d-h)-\delta)!$ and the Vandermonde is $\binom h\delta h!!$
\begin{align*}
c_{d,h,\delta}= & \frac{(h(d-h)-\delta)!}{\prod_{i=1}^\delta(d-h+i-2)!\prod_{i=\delta+1}^h(d-h+i-1)!}\binom h\delta h!!\\
= & \frac{(h(d-h)-\delta)!}{\prod_{i=d-h-1}^{d-h+\delta-2} i!\prod_{i=d-h+\delta}^{d-1} i !}\binom h\delta h!!\\
= & \frac{(h(d-h)-\delta)!}{\frac{(d-h-1)!}{(d-h+\delta-1)!}\frac{d!!}{(d-h)!!}}\binom h\delta h!!\\
= & \delta!((h(d-h)-\delta)!)\binom{d-h+\delta-1}{\delta}\binom{h}{\delta}\frac{h!!(d-h)!!}{d!!}\\
= & (h(d-h))!\frac{\binom{d-h+\delta-1}{\delta}\binom{h}{\delta}}{\binom{h(d-h)}{\delta}}\frac{h!!(d-h)!!}{d!!}.\\
\end{align*}
\end{proof}

\subsection{The coefficient involved in the proof of Theorem \ref{Main}}

\begin{prop}\label{comput2} Denoting $m_{d,\alpha}=d(d+1)/2-\alpha(\alpha+1)/2$. One has\[c_{d,\alpha,\delta}=\frac{2^{m_{d,\alpha}-\delta}(m_{d,\alpha})!}{\binom{m_{d,\alpha}}{\delta}}\frac{\binom{d-\alpha+\delta-1}{\delta}\binom{\alpha+1}\delta\binom{d+\alpha+1}\delta}{\binom{2\alpha+2}\delta}\frac{\alpha!!(d-\alpha)!!(d+\alpha+1)!!}{d!!(2d+1)!!}\left(\prod_{i=\alpha+1}^d(2i-1)!\right).\]
\end{prop}

\begin{proof}

The computation of the coefficient can be continued:
\[c_{d,\alpha,\delta}=\frac{Q_{d,\alpha,\delta}(\underline{b}^*)}{\prod_{i=1}^dg_i'(b_i)},\]
with \[\underline{b}^*=(\underbrace{-(\alpha+1),\dots,-(\alpha-\delta+2)}_{i=1..\delta},\underbrace{-(\alpha-\delta),\dots,-1}_{i=\delta+1..\alpha},\alpha-\delta+1,\underbrace{\alpha+2,\dots,d}_{i=\alpha+2..d}).\]

One has:
\[g'_i(b^*_i)=\begin{cases} (d-\alpha+i-2)! & \textrm{if}\ i\leq\delta,\\
(d-\alpha+i-1)! & \textrm{if}\ \delta<i\leq\alpha,\\
(-1)^\delta \delta!\frac{(d+\alpha+1-\delta)!}{(\alpha-\delta+1)}=(-1)^\delta\frac{(d+\alpha+1)!}{(\alpha-\delta+1)\binom{d+\alpha+1}{\delta}}& \textrm{if}\ i=\alpha+1\\
\frac{(d+i)!}{i} & \textrm{if}\ i>\alpha+1
\end{cases}\]
so the product of their inverse is:

\begin{align*}
\frac{1}{\prod_{i=1}^dg'_i(b^*_i)}= & (-1)^\delta\left(\prod_{i=1}^\delta\frac{1}{(d-\alpha+i-2)!}\right)\left(\prod_{i=\delta+1}^\alpha\frac{1}{(d-\alpha+i-1)!}\right)\\
 & \hspace{2cm}\times\binom{d+\alpha+1}{\delta}\frac{(\alpha-\delta+1)}{(d+\alpha+1)!}\left(\prod_{i=\alpha+2}^d\frac{i}{(d+i)!}\right)\\
\end{align*}
\begin{align*}
 = & (-1)^\delta \frac{(d-\alpha+\delta-1)!}{(d-\alpha-1)!}\left(\prod_{i=1}^\alpha\frac{1}{(d-\alpha+i-1)!}\right)\\
 & \hspace{2cm}\times\binom{d+\alpha+1}{\delta}(\alpha-\delta+1)\frac{d!}{(\alpha+1)!}\left(\prod_{i=\alpha+1}^d\frac{1}{(d+i)!}\right)\\
 = & (-1)^\delta \frac{(d-\alpha+\delta-1)!}{(d-\alpha-1)!}\frac{(d-\alpha)!!}{d!!}\\
 & \hspace{2cm}\times\binom{d+\alpha+1}{\delta}(\alpha-\delta+1)\frac{d!}{(\alpha+1)!}\frac{(d+\alpha+1)!!}{(2d+1)!!}\\
 = & (-1)^\delta \delta!\binom{d-\alpha+\delta-1}{\delta}\binom{d+\alpha+1}{\delta}(\alpha-\delta+1)\frac{d!}{(\alpha+1)!}\frac{(d-\alpha)!!(d+\alpha+1)!!}{d!!(2d+1)!!}.
\end{align*}

The first factor of $Q_{d,\alpha,\delta}$ gives:
\[\prod_{x\in R}(b^*_1+\dots+b^*_d+m-x)=2^{|R|}|R|!=2^{m_{d,\alpha}-\delta}(m_{d,\alpha}-\delta)!\]

The second factor is
\begin{align*}
\prod_{1\leq i<j\leq d}(b^*_j-b^*_i) = & \left(\prod_{1\leq i<j\leq \alpha}\times\prod_{\substack{i=1\\(j=\alpha+1)}}^\alpha\times\prod_{\substack{1\leq i\leq\alpha\\\alpha+1<j\leq d}}\times\prod_{\substack{j=\alpha+2\\(i=\alpha+1)}}^d\times\prod_{\alpha+1<i<j\leq d}\right)(b^*_j-b^*_i)\\
= & \left(\alpha!!\binom\alpha\delta\right)\left(\frac{(2\alpha-\delta+2)!}{(2\alpha-2\delta+2)(\alpha-\delta+1)!}\right)\\
 & \hspace{2cm}\times\left(\prod_{i=1}^\delta\frac{(d+\alpha-i+2)!}{(2\alpha-i+3)!}\prod_{i=\delta+1}^\alpha\frac{(d+\alpha-i+1)!}{(2\alpha-i+2)!}\right)\\
 & \hspace{4cm}\times\left(\frac{(d-\alpha+\delta-1)!}{\delta!}\right) (d-\alpha-1)!!\\
= & \alpha!!\binom\alpha\delta\frac{(2\alpha-\delta+2)!}{(2\alpha-2\delta+2)(\alpha-\delta+1)!}\\
 & \hspace{2cm}\times\frac{(d+\alpha+2)!!(2\alpha-\delta+3)!!}{(d+\alpha-\delta+2)!!(2\alpha+3)!!}\frac{(d+\alpha-\delta+1)!!(\alpha+2)!!}{(d+1)!!(2\alpha-\delta+2)!!}\\
 & \hspace{4cm}\times\binom{d-\alpha+\delta-1}\delta(d-\alpha)!!\\
= & \binom{d-\alpha+\delta-1}\delta\binom\alpha\delta\alpha!!(d-\alpha)!!\\
 & \hspace{2cm}\times\frac{(2\alpha-\delta+2)!}{(2\alpha-2\delta+2)(\alpha-\delta+1)!}\frac{(2\alpha-\delta+2)!}{(d+\alpha-\delta+1)!}\\
 & \hspace{4cm}\times\frac{(d+\alpha+2)!!(\alpha+2)!!}{(d+1)!!(2\alpha+3)!!}\\
\end{align*}
\begin{align*}
= & \binom{d-\alpha+\delta-1}\delta\binom\alpha\delta\alpha!!(d-\alpha)!!\\
 & \hspace{2cm}\times\frac{(2\alpha-\delta+2)!}{(2\alpha-2\delta+2)(\alpha-\delta+1)!}\frac{(2\alpha-\delta+2)!}{(d+\alpha-\delta+1)!}\frac{(d+\alpha+1)!(\alpha+1)!}{(2\alpha+1)!(2\alpha+2)!}\\
 & \hspace{4cm}\times\frac{(d+\alpha+1)!!(\alpha+1)!!}{(d+1)!!(2\alpha+1)!!}\\
= & \frac{(\alpha+1)}{(\alpha-\delta+1)}\frac{\binom{d-\alpha+\delta-1}\delta\binom\alpha\delta\binom{\alpha+1}\delta\binom{d+\alpha+1}\delta}{\binom{2\alpha+2}\delta^2}\frac{\alpha!!(d-\alpha)!!(d+\alpha+1)!!(\alpha+1)!!}{(d+1)!!(2\alpha+1)!!}.
\end{align*}

The last factor is:
\begin{align*}
\prod_{\substack{1\leq i<j\leq d\\ j>\alpha}}(b^*_j+b^*_i)= & \left(\prod_{\substack{i=1\\(j=\alpha+1)}}^\alpha\times\prod_{\substack{1\leq i\leq\alpha\\\alpha+1<j\leq d}}\times\prod_{\substack{j=\alpha+2\\(i=\alpha+1)}}^d\times\prod_{\alpha+1< i<j\leq d}\right)(b^*_j+b^*_i)\\
 = & \left((-1)^\delta\delta!(\alpha-\delta)!\right) \left(\prod_{j=\alpha+2}^d\frac{(j-1)!}{(j-\alpha+\delta-1)(j-\alpha-2)!}\right)\\
 & \hspace{2cm}\times\left(\frac{(d+\alpha-\delta+1)!}{(2\alpha-\delta+2)!}\right) \left(\prod_{i=\alpha+2}^{d-1}\frac{(d+i)!}{(2i)!}\right)\\
= & (-1)^\delta\frac{\alpha!}{\binom\alpha\delta}\frac{\delta!}{(d-\alpha+\delta-1)!}\frac{d!!}{(\alpha+1)!!(d-\alpha-1)!!}\\
 & \hspace{2cm}\times\frac{(d+\alpha-\delta+1)!}{(2\alpha-\delta+2)!}\frac{(2d)!!}{(d+\alpha+2)!!}\prod_{i=\alpha+2}^{d-1}\frac{1}{(2i)!}\\
 = & (-1)^\delta\frac{\binom{2\alpha+2}{\delta}}{\binom\alpha\delta\binom{d+\alpha+1}\delta\binom{d-\alpha+\delta-1}\delta}\frac{d!!(2d)!!}{\alpha!!(d-\alpha)!!(d+\alpha+1)!!}\prod_{i=\alpha+1}^{d-1}\frac{1}{(2i)!}.
\end{align*}

This gives the value of $Q_{d,\alpha,\delta}(\underline{b}^*)$:
\begin{align*}
Q_{d,\alpha,\delta}(\underline{b}^*)= & (-1)^\delta2^{m_{d,\alpha}-\delta}(m_{d,\alpha}-\delta)!\frac{(\alpha+1)}{(\alpha-\delta+1)}\frac{\binom{\alpha+1}\delta}{\binom{2\alpha+2}\delta} \frac{(\alpha+1)!!(2d)!!}{d!(2\alpha+1)!!}\prod_{i=\alpha+1}^{d-1}\frac{1}{(2i)!}\\
 = & (-1)^\delta 2^{m_{d,\alpha}-\delta}(m_{d,\alpha}-\delta)!\frac{(\alpha+1)}{(\alpha-\delta+1)}\frac{\binom{\alpha+1}\delta}{\binom{2\alpha+2}\delta}\frac{(\alpha+1)!!}{d!}\prod_{i=\alpha+1}^d(2i-1)!\\
\end{align*}

And finally

\begin{align*}
c_{d,\alpha,\delta}= & 2^{m_{d,\alpha}-\delta}(m_{d,\alpha}-\delta)!\frac{(\alpha+1)}{(\alpha-\delta+1)}\frac{\binom{\alpha+1}\delta}{\binom{2\alpha+2}\delta}\frac{(\alpha+1)!!}{d!}\prod_{i=\alpha+1}^d(2i-1)!\\
 & \hspace{2cm}\times\delta!\binom{d-\alpha+\delta-1}{\delta}\binom{d+\alpha+1}{\delta}(\alpha-\delta+1)\frac{d!}{(\alpha+1)!}\frac{(d-\alpha)!!(d+\alpha+1)!!}{d!!(2d+1)!!}\\
= & \frac{2^{m_{d,\alpha}-\delta}(m_{d,\alpha})!}{\binom{m_{d,\alpha}}{\delta}}\frac{\binom{d-\alpha+\delta-1}{\delta}\binom{\alpha+1}\delta\binom{d+\alpha+1}\delta}{\binom{2\alpha+2}\delta}\frac{\alpha!!(d-\alpha)!!(d+\alpha+1)!!}{d!!(2d+1)!!}\prod_{i=\alpha+1}^d(2i-1)!\\
\end{align*}

\end{proof}


\begin{thebibliography}{50}

\bibitem{Al} N. Alon, \emph{Combinatorial Nullstellensatz}, Combin. Probab. Comput., \textbf{8} (1999), $7-29$.

\bibitem{ANR1} N. Alon, M.B. Nathanson, I.Z. Ruzsa, \emph{Adding distinct congruence classes modulo a prime}, Am. Math. Monthly \textbf{102} (1995), $250-255$.

\bibitem{ANR2} N. Alon, M.B. Nathanson, I.Z. Ruzsa, \emph{The polynomial method and restricted sums of congruence classes}, J. Number Theory \textbf{56} (1996), $404-417$.

\bibitem{EB} E. Balandraud, \emph{An addition theorem and maximal zero-sum free sets in $\mathbb{Z}/p\mathbb{Z}$}, Israel Journal of Mathematics \textbf{188} (2012), $405-429$.

\bibitem{Ca} A.-L. Cauchy, \emph{Recherches sur les nombres}, J. Ecole Polytech. \textbf{9} (1813), $99-116$.

\bibitem{CZ} K. Cziszter, \emph{Improvements of the Noether bound for polynomial invariants of finite groups}, PhD thesis, CEU Budapest, 2012.\\\url{http://www.etd.ceu.hu/2012/cziszter_kalman-sandor.pdf}

\bibitem{Da1} H. Davenport, \emph{On the addition of residue classes}, J. Lond. Math. Soc. \textbf{10} (1935), $30-32$.

\bibitem{Da2} H. Davenport, \emph{A historical note}, J. Lond. Math. Soc. \textbf{22} (1947), $100-101$.

\bibitem{DH} Dias da Silva, Y. Hamidoune, \emph{Cyclic spaces for Grassman derivatives and additive theory}, Bull. Lond. Math. Soc. \textbf{26} (1994), $140-146$.

\bibitem{Di} G.T. Diderrich, \emph{An addition Theorem for abelian groups of order $pq$}, J. Number Theory \textbf{7} (1975), $33-48$.

\bibitem{EG} P. Erd\"os, R.L. Graham, \emph{Old and New Problems and results in combinatorial Number Theory}, \textbf{28}, L'enseignement math\'ematique, 1980.

\bibitem{EH} P. Erd\"os, H. Heilbronn, \emph{On the Addition of Residue Classes mod $p$}, Acta Arith. \textbf{9} (1964), $149-159$.

\bibitem{KP} R.N. Karasev, F.V. Petrov, \emph{Partitions of nonzero elements of a finite field into pairs}, .

\bibitem{Las} M. Lason, \emph{A generalisation of Combinatorial Nullstellensatz}, The electronic journal of Combinatorics, \textbf{17} (2010), N$32$.

\bibitem{M} M.Michalek, \emph{A short proof of Combinatorial Nullstellensatz}, Am. Math. Monthly, \textbf{117} (2010), $821-823$.

\bibitem{Na} M. B. Nathanson, \emph{Additive number theory: inverse problems and the geometry of sumsets}, GTM \textbf{165}, Springer-Verlag, $1996$.

\bibitem{OPSS} O. Ordaz, A. Philipp, I. Santos, W. Schmid, \emph{On the Olson and the strong Davenport constants}, J. Th\'eor. Nombres Bordeaux, \textbf{23}, (2011), 715-750.

\bibitem{WS} W. Schmid, \emph{Restricted inverse zero-sum problems in groups of rank two}, Q. J. Math., \textbf{63} (2012), $477-487$.

\end{thebibliography}
\end{document}